\documentclass[reqno]{amsart}
\usepackage[utf8]{inputenc}
\usepackage{stackengine}
\usepackage{amssymb}
\usepackage{amsthm}
\usepackage{amsmath}
\usepackage{enumerate}
\usepackage{mathtools}
\usepackage{comment}
\usepackage{todonotes}
\usepackage{xcolor}\mathtoolsset{showonlyrefs}
\usepackage{hyperref}
\usepackage{dsfont}

\newtheorem{thm}{Theorem}[section]
 \newtheorem{cor}[thm]{Corollary}
 \newtheorem{lem}[thm]{Lemma}
 \newtheorem{prop}[thm]{Proposition}

 \newtheorem{defn}[thm]{Definition}

 \newtheorem{rem}[thm]{Remark}

\def\XXint#1#2#3{{\setbox0=\hbox{$#1{#2#3}{\int}$}
     \vcenter{\hbox{$#2#3$}}\kern-.5\wd0}}

\title{Weak derivatives and metric differentiability almost everywhere}

\author{Nikita Evseev}
	\address{Analysis on Metric Spaces Unit, Okinawa Institute of Science and Technology Graduate University, 1919-1 Tancha, Onna-son, Okinawa 904-0495, Japan
	}
	\email{nikita2.evseev@gmail.com}

\thanks{The work was supported by JSPS Grant-in-Aid for Scientific Research No. 25K07038.} 

\begin{document}

\maketitle
\begin{abstract}
It is known that a Lipschitz continuous map from the Euclidean domain to a metric space is metrically differentiable almost everywhere. When the metric space is a Banach space dual to separable, the metric differential has its linear counterpart -- weak* differential. 
However, for an arbitrary metric or Banach space, a Lipschitz map is not necessarily weak* differentiable. This paper introduces an approach based on a concept of weak weak* derivatives. This framework yields a linear representation for the metric differential, allowing for its calculation as the norm of an associated linear operator.
\end{abstract}

\section{Introduction}
The seminal work \cite{K1994} of Kirchheim in 1992 established the metric version of Rademacher's theorem, which asserts that every Lipschitz continuous mapping into a metric space is metrically differentiable almost everywhere.
Metric differentiability is a notion which signifies that the distance between the values of a mapping can be locally approximated by a seminorm on $\mathbb R^n$, denoted $\textnormal{md}(f, x)$.
A mapping $f: \mathbb R^n\to X$ is said to be metrically differentiable at a point $x$ if
\[
\lim_{y\to x}\frac{d(f(y), f(x)) - \textnormal{md}(f, x)(y-x)}{|y-x|} = 0.
\]
Subsequently, Ambrosio and Kirchheim utilized an isometric embedding of an image of Lipschitz mapping into $V^*$, the dual of a separable Banach space  \cite{AK2000}.
They simplified the proof of Kirchheim's result by leveraging the weak* differentiability of a Lipschitz continuous mapping 
$f:\mathbb R^n\to V^*$.
The last means that there is a linear map $wdf(x):\mathbb R^n\to V^*$, weak* differential, such that
\[
\lim_{y\to x}\frac{\langle v, f(y) - f(x) - wdf(x)\cdot(y-x)\rangle}{|y-x|} = 0
\]
for all $v\in V$.
In \cite{AK2000}, it was shown that the metric differential of a Lipschitz continuous $f:\mathbb R^n\to V^*$ can be represented as the norm of the weak* differential, namely for all $\nu\in\mathbb R^n$
\begin{equation}\label{eq:md=wd}
\textnormal{md}(f, x)(\nu) = \| wdf(x)\cdot\nu\|_{V^*}
\end{equation}
for almost every $x\in\mathbb R^n$.
Then for a mapping valued into a metric space $f:\mathbb R^n\to X$ and an isometric embedding $\iota \colon X\to V^*$ 
we have the following realization of the metric differential
\begin{equation}\label{eq:md=wdi}
\textnormal{md}(f, x)(\nu) = \| wd(\iota\circ f)(x)\cdot\nu\|_{V^*}.
\end{equation} 
This raises a crucial question: what if the space $X$ possesses its own linear structure but is not the dual of a separable space? In such settings, formula  \eqref{eq:md=wd} is not directly applicable, as the Lipschitz mapping is not guaranteed to be weak* differentiable. 
Though formula \eqref{eq:md=wdi} is applicable, the embedding $\iota$ will not respect an existing linear structure.
This limitation serves as the primary motivation for our work.
Our objective is to identify a suitable notion of differential that fits into formula \eqref{eq:md=wd}, not only in the case of weak* differentiability, but also in more general settings.
Accordingly, in our main result, we introduce a version of such a differential.
\begin{thm}\label{theorem:main}
Let $\Omega\subset \mathbb R^n$ be an open set and 
$(X, d, z_0)$ be a pointed complete metric space.
Let $f:\Omega\to X$ be a Lipschitz continuous mapping. Then
\begin{enumerate}[(i)]
\item there exists a map from $\Omega$ to the space of bounded linear operators 
$\nabla^{\circ}f:\Omega \to \mathcal L(\mathbb R^n; \left(\textnormal{Lip}_{z_0}(X)\right)^*)$
such that for $\varphi\in \textnormal{Lip}_{z_0}(X)$ and $\nu\in\mathbb R^n$
\begin{equation}
\langle \varphi, \nabla^{\circ}f(x)\cdot\nu \rangle = \nabla\varphi\circ f(x)\cdot\nu \label{eq:nabla0}
\end{equation}
whenever $\varphi\circ f$ is classically differentiable; and
\item there exists a seminorm $\rho$ on $\left(\textnormal{Lip}_{z_0}(X)\right)^*$
such that for $\nu\in\mathbb R^n$
\[
\textnormal{md}(f, x)(\nu) = \rho(\nabla^{\circ}f(x)\cdot\nu)
\]
for almost every $x\in\Omega$.
\end{enumerate}
\end{thm}
Naturally, our notion of differential relates to other linearized interpretations of metric differential,
for example, the one introduced by Gigli, Pasqualetto, and Soultanis in \cite{GPS2020}.
The primary difference lies in the target spaces of the differentials. 
However, both constructions coincide when representing the metric differential of a Lipschitz mapping
(see \cite[Theorem 4.7]{GPS2020} and Lemma \ref{lemma:md=rho}).

In (ii) of Theorem \ref{theorem:main}, $\rho$ appears as a seminorm, which accounts for the possible non-separability of the target space. When the target is separable, it is, in fact, a norm.
Our primary tool is the notion of weak weak* derivatives, which are essentially characterized by condition \eqref{eq:nabla0}. 
Such partial derivatives were introduced in \cite{CE:23} and \cite{CE:2021} and naturally arise for mappings that are absolutely continuous on almost every line, especially when the target space is not a dual of a separable Banach space or lacks the Radon–Nikodým property.
In this paper, we occasionally develop further properties of these derivatives. In particular, it is shown the measurability of the norm and 
\[
\|\partial_j^{\circ} f(x)\|_ {\left(\textnormal{Lip}_{z_0}(X)\right)^*} = 
\lim_{h\to 0} \frac{d(f(x+h\cdot e_j),f(x))}{|h|} \quad \text{ for almost every }  x\in\Omega
\]
Whenever $f$ has a locally integrable partial metric derivative.

\subsection{Organization.}
The paper is organized as follows.
In Section \ref{sec1} we collect necessary results
and definitions concerning the calculus of functions with values in metric spaces.
Then in Section \ref{sec2} we discuss the core object -- weak weak* derivatives. Essentially, in this section, we prove part (i) of  Theorem \ref{theorem:main}.
Section \ref{sec4} is devoted to the metric differentiability and part (ii) ot Theorem \ref{theorem:main} is proved there.
As well, we discuss metric differentiability in the topology of the Sobolev norm.
Finally, in Section \ref{sec5} we consider the case of a linear target.


\section{Preliminaries}\label{sec1}
Throughout this manuscript let $\Omega \subset\mathbb{R}^n$ be an open set, 
$p\in [1,\infty)$, and $(X, d, z_0)$ be a pointed complete metric space, $z_0\in X$ -- fixed point.
We denote by $V^*$ the dual Banach space of $V$. 
For  $v\in V$ and $\Lambda\in V^*$ we use notation $\langle v , \Lambda \rangle :=\Lambda(v)$.

\subsection{Measurability and integrals}
\label{subsec:meas-int}
We call a function $f\colon \Omega \to X$ \emph{measurable} if it is measurable with respect to the Borel $\sigma$-algebra on $X$ and the $\sigma$-algebra of Lebesgue measurable sets on $\Omega$. 
We call $h\colon \Omega \to X$ a \emph{representative} of $f\colon \Omega \to X$  if $h(x)=f(x)$ for almost every $x\in \Omega$.
For $p\in [1,\infty)$ we denote by $L^p(\Omega)$ the space of measurable functions $f\colon \Omega \to \mathbb{R}$ for which the Lebesgue integral $\int_\Omega |f(x)|^p\ \textnormal{d}x$  is finite. 
We say that $f\colon \Omega \to X$ is  \emph{essentially separably valued} if there is a Lebesgue nullset $\Sigma\subset \Omega$ such that $f(\Omega \setminus \Sigma)$ is separable.
We denote by $L^p(\Omega; X)$ the space of measurable essentially separably valued functions 
$f\colon \Omega \to X$ such that $x\mapsto d(f(x), z_0)$  defines an element of $L^p(\Omega)$.
Then we say that $f\in L^1_{\operatorname{loc}}(\Omega; X)$ if $f\in L^1(K; X)$ for any compact $K\subset\Omega$.

\subsection{Lipschitz maps and metric derivatives.}

We denote by $\textnormal{Lip}_{z_0}(X)$ the space of all
real valued Lipschitz functions on X, which vanish at the base point $z_0$.
It is a Banach space with respect to the norm
\[
\operatorname{Lip}(\varphi) := \sup\bigg\{ \frac{|\varphi(x) - \varphi(y)|}{d(x,y)} \colon x,y\in X \text{ and }  x\ne y \bigg\}.
\]

Let $\nu\in \mathbb R^n\setminus\{0\}$. We say that $f\colon \Omega \to X$ has a \emph{metric directional derivative} at $x\in \Omega$ if 
\begin{equation}
    \textnormal{m}\partial_{\nu} f(x):=\lim_{h\to 0} \frac{d(f(x+h\cdot \nu),f(x))}{|h|}\in [0,\infty)
\end{equation}
exists. When $\nu = e_j$ we denote by $\textnormal{m}\partial_j f(x)$ a \emph{$j$-th metric partial derivative}.
As well, for a function $\phi:\Omega\to\mathbb R$ we will use the following notation for the classical directional derivative
\[
 \frac{\partial\phi}{\partial\nu}(x) = \lim_{h\to 0} \frac{\phi(x+h\nu) - \phi(x)}{h}.
\]

A map $f\colon \Omega \to X$ is \emph{absolutely continuous on almost every line segment parallel to $\nu\in \mathbb R^n\setminus\{0\}$} 
if there is a family of lines $L$, which are parallel to $\nu$, such that $\mathcal{L}^n(\Omega\setminus\bigcup L)=0$ and for every $l\in L$ the map $f$ is absolutely continuous when restricted to a line segment contained in $l\cap \Omega$.
Crucially, such a map possesses metric directional derivatives almost everywhere, and these derivatives are themselves measurable functions.
\begin{lem}[{\cite[Lemma 2.4]{CJP22}}]\label{lemma2.4}
Let $f:\Omega\to X$ be a measurable essentially separably valued function.
Suppose that $f$ is absolutely continuous on almost every line segment parallel to $\nu\in \mathbb R^n\setminus\{0\}$,
then $\textnormal{m}\partial_\nu f(x)$ exists at almost every $x\in \Omega$ and the almost everywhere defined function $x\mapsto \textnormal{m}\partial_\nu f(x)$ is measurable.
\end{lem}

A mapping $\Omega\to X$ is \emph{ metrically differentiable} at $x\in\Omega$, if there is a seminorm $\sigma_x:\mathbb R^n\to [0,\infty)$
such that
\begin{equation}\label{eq:def-md}
d(f(y), f(x)) - \sigma_x(y-x) = o(|y-x|).
\end{equation}
In this case, this seminorm is called a \emph{ metric differential} and denoted by $\textnormal{md}(f, x)(\cdot)$.
Recall that function $\sigma:\mathbb R^n\to [0,\infty)$ is a seminorm if $\sigma(\nu + \mu) \leq \sigma(\nu) + \sigma(\mu)$
and $\sigma(t\nu) = |t|\sigma(\nu)$ but can be vanish on a linear subspace of $\mathbb R^n$.
The main examples of metrically differentiable mappings are Lipschitz continuous mappings.  
 \begin{thm}[{\cite[Theorem 2]{K1994}}]
 Let $f:\Omega\to X$ be a Lipschitz continuous mapping, then $f$ is metrically differentiable almost everywhere.
 \end{thm}
Note that if $f$ metrically differentiable at $x$, then $\textnormal{m}\partial_{\nu} f(x) = \textnormal{md}(f, x)(\nu)$.
\subsection{Metric valued Sobolev maps}
We say that  $u:\Omega\to \mathbb R$ belongs to $W^{1,p}(\Omega)$
if $u\in L^p(\Omega)$ and has weak derivatives $\partial_j u \in L^{p}(\Omega)$ for $j=1,2\dots, n$.
The following definition of metric Sobolev space has been given by Reshetnyak in \cite{Reshetnyak97}.
\begin{defn}
\label{definition:WX-space}
The space $W^{1,p}(\Omega; X)$ consists of those $f\in L^p(\Omega; X)$ such that
\begin{enumerate}[(A)]
\item for every $\varphi \in \textnormal{Lip}_{z_0}(X)$  one has $\varphi \circ f \in W^{1,p}(\Omega)$, and
\item there is a function $g \in L^p(\Omega)$ such that for every $\varphi \in \textnormal{Lip}_{z_0}(X)$ one has
\[ |\nabla (\varphi\circ f)(x)| \leq \operatorname{Lip}(\varphi)\cdot g(x) \quad \textnormal{for a.e.\ }x\in \Omega.\]
\end{enumerate}
A function $g$ as in (B) is called a Reshetnyak upper gradient of $f$.
\end{defn}

The following characterization of Sobolev maps by absolute continuity on lines was essentially proved in \cite[Lemma 2.13]{HT:08}
(see also \cite[Lemma 3.3]{CE:2021}, \cite[Theorem 3.1]{CJP22}.)

\begin{prop}
\label{lem:lemmaAC}
Let $f\in W^{1,p}(\Omega; X)$. Then for any $\nu\in\mathbb R^n\setminus\{0\}$
$f$ has a representative $\widetilde{f}$ that is absolutely continuous on almost every line segment parallel to $\nu$. 
Moreover, for every Reshetnyak upper gradient $g$ of $f$ one has
\begin{equation}\label{eq:lemmaACest}
\textnormal{m}\partial_{\nu} \widetilde{f}(x)\leq g(x)
\end{equation}
for almost every $x\in \Omega$.
\end{prop}

\section{Weak weak* derivatives}\label{sec2}
This section studies mappings that are absolutely continuous on lines.
Our primary goal is to establish a rigorous connection between the metric directional derivatives $\textnormal{m}\partial_\nu f$ and derivatives $\partial^\circ_\nu f$, which will be introduced in the following lemma.
These derivatives are referred to as weak weak* derivatives, although they are in fact a particular case of a more general notion defined in Definition \ref{def:w**}.

In the proof of the subsequent lemma, we use the concept of the Banach limit.
 The Banach limit is a continuous linear functional 
$\phi:\ell^{\infty}\to\mathbb R$, which posses the following properties: 
If $a = (a_n)_{n\in\mathbb N}\in \ell^\infty$ and the limit $\lim_{n\to\infty}a_n$ exists
then $\phi(a) =\lim_{n\to\infty}a_n$, 
and if $a_n\geq 0$   for all $n\in\mathbb N$ then  $\phi(a)\geq 0$.
The existence of this functional is guaranteed by the Hahn--Banach theorem (see \cite[Section 7]{Conway1990}).  
We denote the Banach limit by $\phi\text{\,-}\lim\limits_{n\to\infty} a_n$.

\begin{lem}\label{lemma:lip-derivative}
Let $\nu \in \mathbb R^{n}\setminus\{0\}$. Let $f:\Omega\to X$ and for $x\in\Omega$ there exist $C>0$ such that for any $\varphi\in \textnormal{Lip}_{z_0}(X)$
\begin{equation}\label{eq:ww-diff-point}
\limsup_{h\to 0} \frac{|\varphi\circ f(x + h\nu) - \varphi\circ f(x) |}{|h|} \leq C \operatorname{Lip}(\varphi).
\end{equation}
Then there exists a continuous linear functional $\partial^\circ_\nu f(x)\in\left(\textnormal{Lip}_{z_0}(X)\right)^*$
such that
for $\varphi\in \textnormal{Lip}_{z_0}(X)$
\begin{equation*}
\langle \varphi, \partial^\circ_\nu f(x) \rangle = \frac{\partial\varphi\circ f}{\partial\nu}(x)
\end{equation*}
whenever the directional derivative in the right-hand side exists;
and
\[
\| \partial^\circ_\nu f(x) \|_{\left(\textnormal{Lip}_{z_0}(X)\right)^*} \leq C.
\]
In particular, if $f$ has directional metric derivative at $x$ then 
$$
\| \partial^\circ_\nu f(x) \|_{\left(\textnormal{Lip}_{z_0}(X)\right)^*} \leq  \textnormal{m}\partial_\nu f(x).
$$
\end{lem}
\begin{proof}
Fix a sequence $(\delta_n)_{n\in\mathbb N}$ of positive reals such that $\delta_n\to 0$ as $n\to \infty$. 
Define $\partial^\circ_\nu f(x)\colon \textnormal{Lip}_{x_0}(X)\to \mathbb{R}$ by setting
    \begin{equation}
    \langle \varphi, \partial^\circ_\nu f(x) \rangle := 
\phi\text{\,-}\lim\limits_{n\to\infty} \frac{\varphi(f(x+\delta_n \nu)) -\varphi(f(x))}{\delta_n}.
    \end{equation} 
By the linearity of the Banach limit,  $\partial^\circ_\nu f(x)$ is linear.
    Then, from \eqref{eq:ww-diff-point}  we obtain 
\begin{equation*}
|\langle \varphi, \partial^\circ_\nu f(x) \rangle | \leq C \operatorname{Lip}(\varphi).
\end{equation*}
\end{proof}

Let $f:\Omega\to X$ be a function satisfying \eqref{eq:ww-diff-point} of the lemma, and $\psi$ be a Lipschitz mapping from $X$ to another metric space $Y$, with $\psi(z_0) = y_0$.
Then the composition $\psi\circ f$ also satisfies this assumption, and it is easy to see that for any $\varphi\in \textnormal{Lip}_{y_0}(Y)$
\begin{equation}\label{eq:comisometric}
  \langle \varphi, \partial^\circ_\nu (\psi\circ f)(x) \rangle =   \langle \varphi\circ\psi, \partial^\circ_\nu f(x) \rangle.
\end{equation}


\begin{defn}
Let $E\subset X$. 
We call $\varphi_k\in \textnormal{Lip}_{z_0}(X)$ with $\operatorname{Lip}(\varphi_{ik})=1$
an essential gauge sequence for $E$ if 
for any $x,y\in E$
\[
d(x, y) = \sup_{k} |\varphi_k(x) - \varphi_k(y)|.
\]
\end{defn}
\begin{rem}\label{rem:essgauge-sep}
If $E\subset X$ is a separable subspace, then it possesses the following essential gauge sequence.
Let $(x_k)_{k\in\mathbb N}$ be a dense sequence in $E$
then put
\[
\varphi_k(x):= d(x_k, x) - d(x_k,z_0).
\]
\end{rem}

Next we demonstrate that if the postcomposition of a map with Lipschitz functions results in functions that are absolutely continuous on lines, and if their derivatives admit an integrable majorant, then the original mapping possesses an integrable metric derivative.
\begin{thm}\label{theorem:lip-derivative}
Let $\nu\in \mathbb R^n\setminus\{0\}$.
Let $f:\Omega\to X$ be a mapping such that there exists a null set $\Sigma\subset \Omega$ with
$f(\Omega\setminus\Sigma)$ having an essential gauge sequence $(\varphi_k)_{k\in\mathbb N}$
such that $\varphi_k\circ f$ is absolutely continuous on almost every compact line segment, which is contained in $\Omega$
and parallel to $\nu$. 
Also, assume that the function
\[
x\mapsto G(x) := \sup_{k} \bigg| \frac{\partial \varphi_{k}\circ f}{\partial\nu}(x)\bigg|
\]
is locally integrable in $\Omega$.
Then mapping $f$ has a representative $\widetilde f$ which admits a metric directional derivative and
\begin{equation}\label{eq:md=G}
\textnormal{m}\partial_\nu \widetilde f(x) = G(x) \quad \text{for almost every } x\in\Omega.
\end{equation}
\end{thm}
\begin{proof}
For almost every line segment
$l:[a,b]\to\Omega$ that is parallel to $\nu$ one has:
\begin{enumerate}[(i)]
\item  \label{item:i} $G$ is integrable over $l$;
\item  \label{item:ii} $\mathcal{H}^1(l\cap\Sigma) = 0$;
\item \label{item:iii} for every $k \in\mathbb N$ and every $a\leq s\leq t\leq b$
\begin{equation}\label{eq:lemma2.13-condition-c}
|\varphi_{k}\circ f(l(t)) - \varphi_{k}\circ f(l(s))| \leq \int_s^tG(l(\tau)) \ \textrm{d}\tau.
\end{equation}
\end{enumerate}    
The Fubini theorem ensures (\ref{item:i}) and (\ref{item:ii}), while (\ref{item:iii}) follows from absolutely continuity and the defenition of function $G$.
Let $l\colon [a,b] \to \Omega$ be a line segment parallel to $\nu$ for which the properties (\ref{item:i}), (\ref{item:ii}) and (\ref{item:iii}) are satisfied. 
For given $s,t\in l^{-1}(\Omega \setminus \Sigma)$ with $s\leq t$
we have
\begin{equation}\label{eq:R*abs}
d(f(l(t)),f(l(s))) = \sup_k | \varphi_{k}\circ f(l(t)) - \varphi_{k}\circ f(l(s)) | \leq \int_s^tG(l(\tau)) \ \textrm{d}\tau.
\end{equation}
In particular, by properties (\ref{item:i}) and (\ref{item:ii}), and inequality \eqref{eq:R*abs} the restriction of $f$ to $l$ has a unique $\mathcal{H}^1$-representative that is absolutely continuous. The uniqueness implies that these representatives coincide where different line segments overlap. 
Hence, we conclude that $f$ has a representative $\widetilde{f}$ that is absolutely continuous on every compact line segment $l$ that satisfies the properties (\ref{item:i}), (\ref{item:ii}) and (\ref{item:iii}). 
By \eqref{eq:R*abs}  one has 
\begin{equation}\label{eq:limsup}
\limsup_{h\to 0} \frac{d(\widetilde f(x+h\nu),\widetilde f(x))}{h} \leq G(x) 
\quad \text{ for almost every } x\in\Omega.
\end{equation}
Furthermore, $\frac{\partial \varphi_{k}\circ \widetilde f}{\partial\nu}(x) = \frac{\partial \varphi_{k}\circ f}{\partial\nu}(x)$ almost everywhere in $\Omega$.
Then, from inequality $|\varphi_k(\widetilde f(x+h\nu)) - \varphi_k(\widetilde f(x))| \leq d(\widetilde f(x+h\nu),\widetilde f(x))$, for almost every $x\in\Omega$ we have
\begin{equation}\label{eq:liminf}
G(x) \leq \liminf_{h\to 0} \frac{d(\widetilde f(x+h\nu),\widetilde f(x))}{|h|}.
\end{equation}
Therefore, due to \eqref{eq:limsup} and \eqref{eq:liminf}, the metric directional derivative $\textnormal{m}\partial_\nu \widetilde f(x)$ exists almost everywhere and \eqref{eq:md=G} holds. 
\end{proof}

In the case where the mapping is absolutely continuous on lines, it admits a metric directional derivative, as established in Lemma \ref{lemma2.4}. In the following theorem, we show that the metric derivative can be represented as the norm of the derivative defined in Lemma \ref{lemma:lip-derivative}.

\begin{thm}\label{theorem:lip-derivative2}
Let $\nu\in \mathbb R^{n}\setminus\{0\}$ and $f:\Omega\to X$ be a measurable and essentially separably valued function.
Suppose that $f$ is absolutely continuous on almost every compact line segment, which is contained in $\Omega$
and parallel to $\nu$. 
Also, assume that metric directional derivative $\textnormal{m}\partial_\nu f\in L^1_{\operatorname{loc}}(\Omega)$.
Then
\begin{equation} 
\|\partial^\circ_\nu f(x)\|_{\left(\textnormal{Lip}_{z_0}(X)\right)^*} = \textnormal{m}\partial_\nu f(x) 
\end{equation}
for almost every  $x\in\Omega$.
\end{thm}
Particulary, this resolves the question from \cite[Subsection 3.4]{CE:23}
about the measurability of the norm of weak weak* derivatives.
\begin{proof}
By Lemma \ref{lemma2.4} and Lemma \ref{lemma:lip-derivative}
\begin{equation*} 
\|\partial^\circ_\nu f(x)\|_{\left(\textnormal{Lip}_{z_0}(X)\right)^*} \leq \textnormal{m}\partial_\nu f(x) 
\quad \text{ for almost every } x\in\Omega.
\end{equation*}

Let $\Sigma$ be a null-set with $f(\Omega\setminus\Sigma)$ be separable, and
 let $(\varphi_k)_{k\in\mathbb N}$ be an essential gauge sequence for $f(\Omega\setminus\Sigma)$, see Remark \ref{rem:essgauge-sep}. 
 Then due to Theorem \ref{theorem:lip-derivative} and Lemma \ref{lemma:lip-derivative}
\[
\textnormal{m}\partial_\nu f(x) \leq \sup_{k} \bigg| \frac{\partial \varphi_{k}\circ f}{\partial\nu}(x)\bigg|
= \sup_{k} |\langle \varphi_k, \partial^\circ_\nu f(x) \rangle |
\leq \|\partial^\circ_\nu f(x)\|_{\left(\textnormal{Lip}_{z_0}(X)\right)^*}
\]
for almost every  $x\in\Omega$.
Thus, the theorem follows.
\end{proof}

Let $Y$ be another metric space and $F:X\to Y$ be a Lipschitz mapping.
Then, for any $f:\Omega\to X$ as in Theorem \ref{theorem:lip-derivative2}, the composition $F\circ f$ has the same properties.
Therefore 
\begin{equation}\label{eq:der-comp}
\|\partial^\circ_\nu F\circ f(x)\|_{\left(\textnormal{Lip}_{y_0}(Y)\right)^*}
\leq \operatorname{Lip}(F)\|\partial^\circ_\nu f(x)\|_{\left(\textnormal{Lip}_{z_0}(X)\right)^*}.
\end{equation}

\subsection{Weak weak* derivatives and Sobolev spaces}
The characterization of metric valued Sobolev maps in terms of weak weak* derivatives
was done in \cite{CE:23}, \cite{CE:2021}, and \cite{CJP22}. 
Here, we revisit this idea and extend it to a slightly more general setting.
\begin{defn}\label{def:w**}
Let $f\in L^1_{\operatorname{loc}}(\Omega; X)$.
Function $\partial_j f:\Omega\to \left(\textnormal{Lip}_{z_0}(X)\right)^*$ is a $j$-th partial weak weak* derivative of $f$ whenever 
for any $\varphi \in \textnormal{Lip}_{z_0}(X)$ function $x\mapsto \langle \varphi, \partial_j f(x)\rangle$ belongs to $L^1_{\operatorname{loc}}(\Omega)$
and
\begin{equation}
\label{eq:def-w**}
\int_\Omega\frac{\partial \phi}{\partial x_j}(x)\cdot \varphi\circ f(x)\ \textnormal{d}x 
= -\int_\Omega\phi(x)\cdot \langle \varphi, \partial_j f(x) \rangle \ \textnormal{d}x \quad \textnormal{ for every } \phi \in C^\infty_0(\Omega).
\end{equation}
\end{defn}
It turned out that weak weak* derivatives are not unique up to the choice of a representative but instead only up to the choice of a weak* representative
\cite[Lemma 3.4]{CE:23}.
We cannot even guarantee that the norm of a given weak weak* derivative  $\|\partial_j f(\cdot) \|_{\left(\textnormal{Lip}_{z_0}(X)\right)^*}$ is measurable.
The following theorem, however, shows that it is possible to pick a good representative for weak weak* derivative. 
\begin{thm}\label{theorem:Sobolev-space}
Function $f\in W^{1,p}(\Omega;X)$ if and only if 
$f\in L^p(\Omega; X)$ and has a weak weak* partial derivatives $\partial_jf:\Omega\to \left(\textnormal{Lip}_{z_0}(X)\right)^*$
such that $\|\partial_j f(\cdot) \|_{\left(\textnormal{Lip}_{z_0}(X)\right)^*} \in L^p(\Omega)$ for $j=1,\dots n$.
\end{thm}
\begin{proof}
Let $f\in W^{1,p}(\Omega;X)$ and $g\in L^p(\Omega)$ be its Reshetnyak upper gradient.
Thanks to Proposition \ref{lem:lemmaAC}, there is a representative $\widetilde{f}$ that is absolutely continuous on almost every line segment parallel to the $j$-axis.
In particular $\frac{\partial \varphi\circ\widetilde{f}}{\partial x_j}$ is a weak derivative for $\varphi\circ f$ 
for any Lipshitz function $\varphi:X\to\mathbb R$. 
Then by Lemma \ref{lemma:lip-derivative} there exists a linear functional $\partial^\circ_j \widetilde{f}(x)$
such that for $\varphi\in \textnormal{Lip}_{z_0}(X)$
\[
\frac{\partial \varphi\circ\widetilde{f}}{\partial x_j}(x) = \langle \varphi, \partial^\circ_j \widetilde{f}(x) \rangle \quad \textnormal{for almost every } x\in\Omega.
\]
Therefore $\langle \varphi, \partial^\circ_j\widetilde{f}(x) \rangle$ is a weak derivative of $\varphi\circ f$  for every $\varphi\in \textnormal{Lip}_{z_0}(X)$ and 
\[
|\langle \varphi, \partial^\circ_j \widetilde{f}(x) \rangle | = |\partial_j (\varphi\circ f)(x) |  \leq  \operatorname{Lip}(\varphi)\cdot g(x)
\]
for almost every  $x\in\Omega$. This means that $\partial^\circ_j \widetilde{f}$ is a partial weak weak* derivative of $f$.
Moreover,  due to Theorem \ref{theorem:lip-derivative2} function 
$x\mapsto \|\partial^\circ_j \widetilde{f}(x)\|_{\left(\textnormal{Lip}_{z_0}(X)\right)^*}$
is measurable, and we have 
\[
\int_\Omega \|\partial^\circ_j \widetilde{f}(x)\|_{\left(\textnormal{Lip}_{z_0}(X)\right)^*}^p \ \textnormal{d}x
\leq \int_\Omega (g(x))^p \ \textnormal{d}x < \infty \quad \textnormal{ for } j=1,\dots n.
 \]
 
 Now suppose that $f\in L^p(\Omega; X)$ and has weak weak* partial derivatives $\partial_jf$ 
 such that $\|\partial_jf(\cdot) \|_{\left(\textnormal{Lip}_{z_0}(X)\right)^*} \in L^p(\Omega)$ for $j=1,\dots n$. 
  For any $\varphi\in \textnormal{Lip}_{z_0}(X)$ we have $|\varphi\circ f(x)| \leq \operatorname{Lip}(\varphi)d(f(x), z_0)$,
  therefore $\varphi\circ f \in L^p(\Omega)$.
  By definition of weak weak* derivatives $\langle \varphi, \partial_jf(x) \rangle$ is a weak derivative of $\varphi\circ f$,
  and it belongs to $L^p(\Omega)$ because 
\[
|\partial_j (\varphi\circ f)(x)| =  |\langle \varphi, \partial_j f(x) \rangle | \leq \operatorname{Lip}(\varphi) \|\partial_jf(x)\|_{\left(\textnormal{Lip}_{z_0}(X)\right)^*}.
\]
So $\varphi\circ f \in W^{1,p}(\Omega)$.
Finally, we have
\begin{equation}
\label{eq:second-norm}
|\nabla (\varphi\circ f)(x)|= \left( \sum_{j=1}^n  \left| \partial_j (\varphi\circ f)(x)\right|^2\right)^{1/2}
\leq  \operatorname{Lip}(\varphi)\left( \sum_{j=1}^n \|\partial_jf(x)\|^2_{\left(\textnormal{Lip}_{z_0}(X)\right)^*}\right)^{1/2}
\end{equation}
for almost every $x\in\Omega$.
Setting $g(x):=\left( \sum_{j=1}^n \|\partial_jf(x)\|^2_{\left(\textnormal{Lip}_{z_0}(X)\right)^*}\right)^{1/2}$, we conclude that $f\in W^{1,p}(\Omega;X)$.
\end{proof}

\begin{lem}\label{lemma:der-same}
Let $f_1,f_2 \in W^{1,p}(\Omega; X)$ be
such that $f_1=f_2$ almost everywhere in a measurable set $E\subset\Omega$.
Then
\[
\|\partial^\circ_j \widetilde f_1(x)\|_{\left(\textnormal{Lip}_{z_0}(X)\right)^*} = \|\partial^\circ_j \widetilde f_2(x)\|_{\left(\textnormal{Lip}_{z_0}(X)\right)^*} \quad j=1\dots n
\]
for almost every $x\in E$. Here $\widetilde f_1$ and $\widetilde f_1$  are epresentatives that are absolutely continuous on almost every line segment parallel to the $j$-axis.
\end{lem}
\begin{proof}
For almost every line segment $l$ parallel to $j$-th axis with $\mathcal{H}^1(l\cap E) > 0$ we have $f_1=f_2$ on $l\cap E$.
Therefore metric directional derivatives $\textnormal{m}\partial_j f_1(x)$ and $\textnormal{m}\partial_j f_2(x)$ coincide for almost every $x\in E$.
Thus, by Theorem \ref{theorem:lip-derivative2}, the lemma follows.
\end{proof}

\begin{lem}\label{lemma:SobLip}
Let $p\in[1,\infty)$, $\Omega$ be a bounded domain, and map $f \in W^{1,p}(\Omega; X)$ be such that 
\[
\sup_{x,y\in\Omega} d(f(x), f(y)) \leq R <\infty.
\]
Then for every $\varepsilon>0$ there exist a measurable set $A_{\varepsilon}\subset\Omega$
so that $|\Omega\setminus A_\varepsilon|<\varepsilon$ and
the restriction $f|_{A_\varepsilon}:A_\varepsilon\to X$ is Lipschitz continuous. 
Moreover 
\begin{equation}\label{eq:lip-measure}
\lim_{\varepsilon\to 0} \left( \operatorname{Lip}(f|_{A_\varepsilon})\right)^p|\Omega\setminus A_\varepsilon| = 0.
\end{equation}
\end{lem}
\begin{proof}
Fix $\varepsilon>0$ and a compactly embedded domain $D\subset\Omega$ such that $|\Omega\setminus D|<
\min\{\frac{\varepsilon}{2}, \varepsilon^2\}$.
Choose function $\eta\in C^{\infty}_0(\mathbb R^n)$ with $\eta(x)=1$ when $x\in D$ and $\eta(x) = 0$ when $x\in \mathbb R^n\setminus \Omega$.
Let $\{\varphi_k\}_{k\in\mathbb N}$ be an essential gauge sequence for $f(\Omega\setminus\Sigma)$, where $|\Sigma|=0$.
Then function $\eta\cdot\varphi_k\circ f$ can be extended by zero to belong $W^{1,p}(\mathbb R^n)$.
Therefore, we can apply a pointwise estimate via the maximal function, i.e. 
\[
|\varphi_k\circ f(x) - \varphi_k\circ f(y)| \leq
C|x-y|(M|\nabla (\eta\cdot\varphi_k\circ f)|(x) + M|\nabla (\eta\cdot\varphi_k\circ f)|(y))
\]
for almost every $x,y\in D$. Then we have the following estimate
\begin{align}
|\nabla( \eta\cdot\varphi_k\circ f)(x)| = | \varphi_k\circ f(x)\cdot\nabla \eta(x) + \eta(x)\cdot\nabla( \varphi_k\circ f)(x)|\\
\leq R\cdot|\nabla \eta(x)| + \eta(x) \cdot  \|\nabla^\circ f(x)\|_{\left(\textnormal{Lip}_{z_0}(X)\right)^*} =: h(x),
\end{align}
where 
\[
\|\nabla^\circ f(x)\|_{\left(\textnormal{Lip}_{z_0}(X)\right)^*} =  \left(\sum_{j=1}^n \|\partial^\circ_j f(x)\|_{\left(\textnormal{Lip}_{z_0}(X)\right)^*}^2 \right)^{\frac{1}{2}}.
\]
Note that $h\in L^p(\mathbb R^n)$. Subsequently, we have 
\[
d(f(x),f(y)) \leq C|x-y|(Mh(x) + Mh(y)) 
\]
for almost every $x,y\in D$. Let $t>0$, denote $E_t:= \{x\in \mathbb R^n \colon Mh(x) \geq t \}$, then 
\[
|E_t| \leq c\|h\|_{L^p(\mathbb R^n)} \frac{1}{t^p}.
\]
Therefore, taking
\[
t\geq \bigg(\frac{2c}{\varepsilon}\|h\|_{L^p(\mathbb R^n)} \bigg)^{\frac{1}{p}} \quad \text{ and }\quad A = D\setminus E_t
\]
we obtain that $f|_A$ is Lipschitz contnuous with $\operatorname{Lip}(f|_A)\leq 2Ct$ and $|\Omega\setminus A|\leq \varepsilon$.
Finally, \eqref{eq:lip-measure} also follows from the properties of the maximal function, see for example the proof of \cite[Theorem 4.9]{KLV2021}.
\end{proof}

Though it might be natural to expect the Sobolev space $W^{1,p}(\Omega; X)$ to be a metric space, it is somewhat problematic to impose a distance.
 One available approach employs an isometric embedding of a metric space $X$ into some Banach space. 
 However, Haj{\l}asz in \cite{H2011} esteblished that some properties of such an induced metric may drastically depend on the isometric embedding.
 On the other hand, when $X$ is a Banach space, the corresponding Sobolev space turns out to be a Banach space itself.

Now, let $V$ be a Banach space. Because $V^*\subset \textnormal{Lip}_{0}(V)$ we can define a linear functional $\partial^{**}_{\nu}f(x)\in V^{**}$
as
\[
\big\langle v^*, \partial^{**}_\nu f(x)\big\rangle_{V^*,V^{**}} 
:= \big\langle v^*, \partial^\circ_\nu f(x) \big\rangle_{\textnormal{Lip}_{0}(V), \left(\textnormal{Lip}_{0}(V)\right)^*}.
\]
If $f$ and $g$ are two functions from $\Omega$ to $V$, that statisfy \eqref{eq:ww-diff-point}. Then, by definition of $\partial^\circ_\nu$ and the linearity of the Banach limit, we have that
\[
\partial^{**}_\nu(f +g )(x) = \partial^{**}_\nu f (x) + \partial^{**}_\nu g(x).
\]
From the proof of Theorem \ref{theorem:Sobolev-space}, we conclude that $\partial^\circ_j \widetilde{f}(x)$ represents a \textit{canonical} weak weak* partial derivative.
Then, for a Sobolev mapping $f\in W^{1,p}(\Omega; V)$ we assume that $\partial^{**}_j f(x): =\partial^{**}_j \widetilde{f}(x)$.
With this in mind, we define a \textit{norm}
\[
   \|f\|_{W^{1,p}}:=\left(\int_\Omega \|f(x)\|_{V}^{p}\ \textnormal{d}x\right)^{\frac{1}{p}} + 
   \sum_{j=1}^n \left(\int_\Omega \|\partial^{**}_j f(x) \|^p_{V^{**}} \textnormal{d}x\right)^{\frac{1}{p}}.
\]

\begin{lem}\label{lemma:lipapprox}
Let $f:\Omega\to V$ be a map belonging to the Sobolev space $W^{1,p}(\Omega; V)$, $p\in [1,\infty)$.
Then for every $\varepsilon>0$ there exists a Lipschitz continuous map $F:\Omega\to V$ such that
\[
|\{x\in\Omega \colon f(x) \ne F(x)\} | < \varepsilon \quad\text{ and }\quad \|f - F \|_{W^{1,p}(\Omega; V)}<\varepsilon.
\]
\end{lem}
\begin{proof}
For $R>0$ define  by
\[
\pi_R(v) = 
\begin{cases}
v, &\text{ if } \|v\|_V \leq R; \\
\frac{Rv}{\|v\|_V} , &\text{ if } \|v\|_V > R.
\end{cases}
\]
Observe that $\pi_R:V\to V$ is a Lipschitz mapping with $\operatorname{Lip}(\pi_R) \leq 2$.
Denote
\[
E_R:= \{x\in \Omega \colon \|f(x)\|> R\}.
\]
Then, applying Lemma \ref{lemma:der-same} and \eqref{eq:der-comp} we obtain 
\[
\|\pi_R\circ f  - f\|_{W^{1,p}} 
\leq 3\left(\int_{E_R} \|f(x)\|_{V}^{p}\ \textnormal{d}x\right)^{\frac{1}{p}} + 
   3\sum_{j=1}^n \left(\int_{E_R} \|\partial^{**}_j f(x) \|^p_{V^{**}} \textnormal{d}x\right)^{\frac{1}{p}}. 
\]
By the absolute continuity of the Lebesgue integral, we can choose such $R>0$ that
\[
|E_R|\leq \frac{\varepsilon}{2} \quad\text{ and }\quad  \|\pi_R\circ f  - f\|_{W^{1,p}}  \leq \frac{\varepsilon}{2}.
\]
Fix this $R$.
By Lemma \ref{lemma:SobLip}, for $\epsilon>0$ to be specified later, there is a set $A_{\epsilon}$ such that 
$\pi_R\circ f|_{A_\epsilon}:A_{\epsilon}\to X$ is Lipschitz continuous and $|\Omega\setminus A_\epsilon|< \epsilon$. 
Thanks to  \cite[Theorem 2]{JLS1986} there is an extention $F_\epsilon:\Omega\to V$ with $\operatorname{Lip}(F_\epsilon) \leq c(n)\operatorname{Lip}(\pi_R\circ f|_{A_\epsilon})$ and we can assume that $\|F_\epsilon(x)\|_V\leq R$ for $x\in\Omega$.
Thus, by Lemma \ref{lemma:der-same} we have
\begin{align}
\|\pi_R\circ f - F_\epsilon \|_{W^{1,p}} \leq 2R|\Omega\setminus A_\epsilon|^{\frac{1}{p}} &+ 2\sum_{j=1}^n \left(\int_{\Omega\setminus A_\epsilon} \|\partial^{**}_j f(x) \|^p_{V^{**}} \textnormal{d}x\right)^{\frac{1}{p}}\\
&+ n\operatorname{Lip}(F_{\epsilon})|\Omega\setminus A_\epsilon|^{\frac{1}{p}}.
\end{align}
Then, by the absolute continuity of the Lebesgue integral and \eqref{eq:lip-measure} we can choose $\epsilon\in(0,\frac{\varepsilon}{2})$, so that
\[
 \|\pi_R\circ f - F_\epsilon \|_{W^{1,p}} \leq \frac{1}{\varepsilon}.
\]
Finaly, note that $f(x) = F(x)$ for $x\in A_{\epsilon}\setminus E_R$.
This concludes the proof of the lemma.
\end{proof}
Note that the obtained approximation by Lipschitz mappings is compatible with that of classical Sobolev spaces (\cite[Theorem 4.9]{KLV2021}).
For the case when mappings are defined on a metric space, see \cite[Theorem 8.2.1]{HKST2015} and \cite[Lemma 13]{Haj2009MathAnn}.

\section{Metric differentiability}\label{sec4}
The following lemma provides the central link between our framework and classical theory, expressing the metric differential of a Lipschitz continuous mapping in terms of the derivatives developed in the previous section. 
\begin{lem}\label{lemma:md=rho}
 Let $f:\Omega\to X$ be a Lipschitz continuous map.
 Then there exists a seminorm $\rho$ on $\left(\textnormal{Lip}_{z_0}(X)\right)^*$ such that  
 for $\nu\in\mathbb R^n$
 \begin{equation}\label{eq:md=rho}
 \textnormal{md}(f, x)(\nu) = \rho(\nabla^of(x)\cdot\nu )  \quad \textnormal{for almost every } x\in\Omega,
 \end{equation}
 where we denote $\nabla^of(x)\cdot\nu  = \sum_{j=1}^n\partial^o_jf(x)\nu_j$.
 \end{lem}

\begin{proof}
Fix a countable dense subset $H\subset\mathbb R^n$.
Let $\{\varphi_k\}_{k\in\mathbb N}$ be an essential gauge sequence for $f(\Omega)$.
Then due to Lemma \ref{lemma:lip-derivative} and Theorem \ref{theorem:lip-derivative},
almost everywhere in  $\Omega$ for each $\nu\in H$ the following hold
\begin{equation}\label{eq:H-norm-md}
\sup_{k} \bigg| \frac{\partial \varphi_{k}\circ f}{\partial\nu}(x)\bigg| = \|\partial^\circ_\nu f(x)\|_{\left(\textnormal{Lip}_{z_0}(X)\right)^*} =  \textnormal{md}(f, x)(\nu) 
\end{equation}
and
\begin{equation}\label{eq:directional-lip}
\frac{\partial \varphi_{k}\circ f}{\partial\nu}(x) =  \langle \varphi_k, \partial^\circ_\nu f(x) \rangle \quad \textnormal{ for } k\in\mathbb N.
\end{equation}
Now let $\mu\in\mathbb R^n$ and $(\nu_i)_{i\in\mathbb N}$ be a sequence from $H$ converging to $\mu$. 
Then 
\[
| \langle \varphi_k, \partial^\circ_\mu f(x) \rangle -  \langle \varphi_k, \partial^\circ_{\nu_i} f(x) \rangle | \leq \operatorname{Lip}(f)\cdot |\mu - \nu_i|,
\]
from where we conclude
\begin{equation}\label{eq:lip-derivative-mu}
 \langle \varphi_k, \partial^\circ_\mu f(x) \rangle = \lim_{i\to\infty}   \langle \varphi_k, \partial^\circ_{\nu_i} f(x) \rangle
\end{equation}
and 
\begin{equation}\label{eq:norm-mu}
\sup_k|  \langle \varphi_k, \partial^\circ_\mu f(x) \rangle| = \lim_{i\to\infty} \sup_k|  \langle \varphi_k, \partial^\circ_{\nu_i} f(x) \rangle|. 
\end{equation}
From \eqref{eq:H-norm-md}, \eqref{eq:directional-lip} and \eqref{eq:norm-mu} we obtain that 
\[
\sup_k|  \langle \varphi_k, \partial^\circ_\nu f(x) \rangle| =  \textnormal{md}(f, x)(\nu) 
\]
holds for all $\nu\in\mathbb R^n$.
The compositions $\varphi_{k}\circ f$ are Lipschitz continuous. 
Therefore, they are differentiable almost everywhere in $\Omega$
and
\[
\frac{\partial \varphi_{k}\circ f}{\partial\nu}(x) = \nabla (\varphi_{k}\circ f)(x) \cdot \nu
\quad \textnormal{ for } k\in \mathbb N \textnormal{ and } \nu \in \mathbb R^n.
\] 
Thus we have
\begin{align}
\langle \varphi_k, \partial^\circ_\nu f(x) \rangle
= \frac{\partial \varphi_{k}\circ f}{\partial\nu}(x) 
= \sum_{j=1}^n \frac{\partial \varphi_{k}\circ f}{\partial x_j}(x) \cdot\nu_j \\
 = \sum_{j=1}^n \langle \varphi_k, \partial^o_jf(x) \rangle \cdot\nu_j 
  = \langle \varphi_k, \nabla^of(x)\cdot\nu \rangle.
\end{align}
Finally, the desired seminorm is given by
\[
\rho(w) = \sup_k|  \langle \varphi_k, w \rangle|.
\]
\end{proof}
It is important to emphasize that our proof of relation  \eqref{eq:md=rho} relies on both the metric differentiability and the Lipschitz continuity of the mapping. 
This raises the natural question of whether the same relation holds under weaker assumptions.

\subsection{Differentiability with respect to the Sobolev norm}
Metric differentiability can be interpreted as approximation by a seminorm with respect to the uniform norm. Replacing the uniform norm with the Sobolev norm leads to a different notion of metric differentiability. In this subsection, we show that every Sobolev mapping is metrically differentiable with respect to the topology of $W^{1,p}$.
This property might have various consequences, in particular, in the case $p>n$, we can obtain metric differentiability almost everywhere of a Sobolev mapping (Corollary \ref{cor:p-n}). 
The original idea of this approach goes back Reshetnyak's work \cite{R1969}, however, in the proof of Theorem \ref{theorem:metricSob} we follow methods developed by Vodopyanov in \cite{V:2003}.

\begin{defn}
A seminorm $\sigma$ on $\mathbb R^n$ is called a metric differential of a map $f:\Omega\to X$
at a point $x\in\Omega$ in the topology of $W^{1,p}$, $p\in[1,\infty)$, if
\[
\bigg\| \frac{d(f(x+h\cdot), f(x))}{h} - \sigma(\cdot) \bigg\|_{W^{1,p}(B)} = o(1) 
\quad \textnormal{ as } h\to 0,
\]
where $B$ is the unit ball in $\mathbb R^n$.
\end{defn}

\begin{lem}\label{lemma:LipDiff}
Let $f:\Omega\to X$ be a Lipschitz continuous map.
Then, whenever it exists, $\textnormal{md}(f, x)$ is also a metric differential in the topology of $W^{1,p}$, $p\in[1,\infty)$.
\end{lem}

\begin{proof}
Let $x$ be a point of metric differentiability of $f$. We need to verify the following relations
\begin{equation}\label{eq:L_p-conv}
\int_B \bigg|\frac{d(f(x+h\nu), f(x))}{h} - \textnormal{md}(f, x)(\nu)\bigg|^p  \ \textnormal{d}\nu = o(1)
\end{equation}
and
\begin{equation}\label{eq:Der-conv}
\int_B \bigg|\frac{\partial}{\partial \nu_j}\frac{ d(f(x+h\nu), f(x))}{h} - \frac{\partial}{\partial \nu_j}\textnormal{md}(f, x)(\nu)\bigg|^p  \ \textnormal{d}\nu = o(1)
\end{equation}
Note that \eqref{eq:L_p-conv} holds due to uniform convergence in \eqref{eq:def-md}. 
Now prove \eqref{eq:Der-conv}. 
Let  $(h_k)_{k\in\mathbb N}$ be a sequence such that $h_k\to 0$. Then
\[
\textnormal{md}(f, x)(\nu) = \lim_{k\to \infty} \frac{d(f(x+h_k\nu), f(x))}{h_k}.
\]
Consider 
\[
F(r,h_k) = \frac{d(f(x+h_k(\nu+re_j)), f(x)) - d(f(x+h_k\nu), f(x))}{rh_k}.
\]
Then, due to metric differentiability in $x$ we obtain that
\[
\lim_{k\to \infty} F(r,t) = \frac{\textnormal{md}(f, x)(\nu + re_j) - \textnormal{md}(f, x)(\nu)}{r}
\]
uniformly with respect to $r$ and $\nu$. 
On the other hand, function $\nu\mapsto \frac{d(f(x+h_k\nu), f(x))}{h_k}$ is Lipschitz continuous.
Therefore 
\[
\lim_{r\to 0} F(r,h_k) = \frac{\partial}{\partial \nu_j}\frac{d(f(x+h_k\nu), f(x))}{h_k}
\]
holds for all $k\in\mathbb N$ and $\nu\in B\setminus\bigcup_{k}\Sigma_k$, with $|\Sigma_k| = 0$.
Therefore, we can interchange the order of limits and derive that
\[
\frac{\partial}{\partial \nu_j}\textnormal{md}(f, x)(\nu) = \lim_{k\to\infty} \frac{\partial}{\partial \nu_j}\frac{d(f(x+h_k\nu), f(x))}{h_k}
\quad \textnormal{for almost every } \nu \in B.
\]
By the dominated convergence theorem and the arbitrariness of the sequence $(h_k)_{k\in\mathbb N}$ 
we conclude \eqref{eq:Der-conv}.
\end{proof}

In the proof of the following theorem, we use an isometric embedding into a Banach space. That might interfere with the idea of developing intrinsic methods presented in the paper. As an excuse, one can say that the fact of metric differentiability does not depend on embedding. The reason we need a Banach space is that we use (in Lemma \ref{lemma:lipapprox}) 
the extension of a Lipschitz mapping defined on Euclidean space. 

\begin{thm}\label{theorem:metricSob}
Every map in the Sobolev space 
$W^{1,p}(\Omega; X)$, $p\in [1, \infty)$
is almost everywhere metrically differentiable in the topology of $W^{1,p}$.
There exists a seminorm $\rho$ on $\left(\textnormal{Lip}_{z_0}(X)\right)^*$ such that  
  $\mathbb R^n \ni \nu\mapsto \rho(\nabla^of(x)\cdot\nu ) $
is the metric differential of a map $f$
at a point $x\in\Omega$ in the topology of $W^{1,p}$. 
\end{thm}
\begin{proof}
Let $k:X\to V$ be an isometric embedding with $k(z_0)=0$. Then $k\circ f \in W^{1,p}(\Omega; V)$.
Let $\Sigma$ be a null set such that the image $k\circ f(\Omega\setminus\Sigma)$ is separable in $V$ and let $\{\phi_k\}_{k\in\mathbb N}$ be an essential gauge sequence for 
$k\circ f(\Omega\setminus\Sigma)$.
Define seminorm
$
\rho(w) = \sup_k|  \langle \varphi_k, w \rangle|
$
for $w\in \left(\textnormal{Lip}_{0}(V)\right)^*$.

By Lemma \ref{lemma:lipapprox} for each $\varepsilon>0$ there is a measurable set $A_\varepsilon\subset\Omega$
and a Lipschitz continuous map $F_\varepsilon:\Omega\to X$ such that $|\Omega\setminus A_\varepsilon|<\varepsilon$
and $F_\varepsilon(x) = k\circ f(x)$ for $x\in A_\varepsilon$. 
Let $x\in A_\varepsilon$ be a density point, a point of metric differentiability of $F_\varepsilon$,
and a point where partial derivatives $\frac{\partial \phi_k\circ k\circ  f}{\partial x_j}(x)$ and $\frac{\partial \phi_k\circ F_\varepsilon}{\partial x_j}(x)$
exist for $j=1,\dots n$ and $k\in \mathbb N$.
Then $\langle \phi_k, \nabla^{\circ} (k\circ f)(x) \rangle = \langle \phi_k, \nabla^{\circ}F_\varepsilon(x) \rangle $,
and in particular, $\rho(\nabla^{\circ}k\circ f(x)\cdot\nu ) = \rho(\nabla^{\circ}F_\varepsilon (x)\cdot\nu )$.
For $h>0$ denote $D_h = \{\nu\in B \colon x + h\nu \in A_\varepsilon\}$.
Because $x$ is a density point, we have 
\begin{equation}\label{eq:BD-0}
\lim_{h\to 0} |B\setminus D_h| = 0.
\end{equation}
Consider functions
\[
\eta_h(\nu) = \frac{\|k\circ f(x+h\nu) - k\circ f(x)\|_V}{h} - \rho(\nabla^{\circ}(k\circ f)(x)\cdot\nu )
\]
and
\[
\psi_h(\nu) = \frac{\| F_\varepsilon(x+h\nu) -  F_\varepsilon(x)\|_V}{h} - \rho(\nabla^{\circ}F_\varepsilon (x)\cdot\nu ).
\]
Note that $\eta_h(\nu) = \psi_h(\nu)$ for $\nu\in\ D_h$, and
due to Lemma \ref{lemma:LipDiff} 
\begin{equation}\label{eq:psi-0}
\lim_{h\to 0} \| \psi_h\|_{W^{1,p}(B)} = 0.
\end{equation} 
Now for $\varepsilon_1>0$ find Lipschitz function $\xi_h:B\to \mathbb R$ such that 
$\eta_h(\nu) = \xi_h(\nu)$ on measurable set $E_h\subset B$, $|B\setminus E_h| < \varepsilon_1$, and
$\|\eta_h - \xi_h \|_{W^{1,p}(B)}\leq\varepsilon_1$ (Lemma \ref{lemma:lipapprox}).
Then $\psi_h(\nu) = \xi_h(\nu)$ for all $\nu\in D_h\cap E_h$ and $\nabla\psi_h(\nu) = \nabla\xi_h(\nu)$ for almost every $\nu\in D_h\cap E_h$.
Show that the Lipschitz constants $\operatorname{Lip}(\psi_h)$ and $\operatorname{Lip}(\xi_h)$ do not depend on $h$.
Indeed
\[
\operatorname{Lip}(\psi_h) \leq \operatorname{Lip}(F_\varepsilon) + \bigg(\sum_{j=1}^n \|\partial^o_jf(x) \|^2_{\left(\textnormal{Lip}_{0}(X)\right)^*} \bigg)^{\frac{1}{2}}.
\]
While $\operatorname{Lip}(\xi_h)$ depends only on $|\nabla \eta_h (\nu)|$, for which we have the following estimate
\[
|\nabla \eta_h(\nu)| \leq g(x + h\nu) + \bigg(\sum_{j=1}^n \|\partial^o_jf(x) \|^2_{\left(\textnormal{Lip}_{0}(X)\right)^*} \bigg)^{\frac{1}{2}}.
\]
For fixed $\varepsilon_1$ it implies $\operatorname{Lip}(\xi_{h_1}) \leq \operatorname{Lip}(\xi_{h_2})$, when $h_1<h_2$.
So, we can assume that $\operatorname{Lip}(\xi_h) \leq \operatorname{Lip}(\xi_{h_0})$ for all $h\leq h_0$.
To estimate the norm $\|\xi_h - \psi_h \|_{W^{1,p}(B)}$ we note that $B\subset (B\setminus E_h)\cup (B\setminus D_h)\cup (E_h\cap D_h)$
and derive
\begin{align*}
\|\xi_h - \psi_h \|_{W^{1,p}(B)} \leq &C\cdot\varepsilon_1^{\frac{1}{p}} + C|B\setminus D_h|^{\frac{1}{p}} + 0\\
+ &C'\cdot\varepsilon_1^{\frac{1}{p}} + C'|B\setminus D_h|^{\frac{1}{p}} + 0,
\end{align*}
where $C = \sup_{B}|\xi_h(\nu) - \psi_h(\nu)| $ and $C' = 2\max\{\operatorname{Lip}(\psi_h), \operatorname{Lip}(\xi_h) \}$,
both constants do not depend on $h$.
Finally, from \eqref{eq:BD-0}, \eqref{eq:psi-0} and 
\[
\|\eta_h \|_{W^{1,p}(B)} \leq \|\eta_h -\xi_h \|_{W^{1,p}(B)} + \|\xi_h - \psi_h \|_{W^{1,p}(B)} + \|\psi_h \|_{W^{1,p}(B)}
\]
we have
\[
\limsup_{h\to 0 } \|\eta_h \|_{W^{1,p}(B)} \leq \varepsilon_1 + (C+ C')\varepsilon_1^{\frac{1}{p}}.
\]
Due to the arbitrariness of $\varepsilon_1$ and $\varepsilon$, we conclude that $k\circ f$ is almost everywhere metrically differentiable in the topology of $W^{1,p}$
with metric differential $\nu\mapsto \sup_k|  \langle \varphi_k, \nabla^{\circ} (k\circ f)(x)\cdot\nu \rangle|$.
It remains to note that $d(f(x+h\nu), f(x)) = \|k\circ f(x+h\nu) - k\circ f(x)\|_V$. Therefore mapping $f$ is also almost everywhere metrically differentiable in the topology of $W^{1,p}$
and with property \eqref{eq:comisometric} its corresponding metric differential could be writen as 
$\nu\mapsto \sup_k|  \langle \phi_k\circ k, \nabla^{\circ} f(x)\cdot\nu \rangle|$.
\end{proof}

The following corollary provides us with examples of mappings that are not necessarily Lipschitz continuous, but metrically differentiable almost everywhere.
\begin{cor}\label{cor:p-n}
Every function in the Sobolev space $W^{1,p}(\Omega; X)$, $p>n$, 
has a continuous representative, which is almost everywhere metrically differentiable.
\end{cor}
\begin{proof}
Thanks to \cite[Theorem 6.2]{Reshetnyak97}, mapping $f$ has a continuous representative.
Then, by the Sobolev embedding theorem for $W^{1,p}(B)$, the corollary follows.
\end{proof}
We note that the same conclusion was reached in \cite[Corollary 3.4]{K2007}, albeit using different methods.

\section{The case of linear target}\label{sec5}
When $X$ is a linear space, a more detailed description is available.
Although we cannot establish weak* differentiability, it is still possible to have the following differentiability-like results.

\begin{thm}\label{thm:Vtarget}
Let $V$ be a Banach space and let $f:\Omega\to V$ be a Lipschitz mapping.
Then there exists a map $\nabla^{**}f :\Omega \to \mathcal B(\mathbb R^n; V^{**})$
such that for any separable subset $D^*\subset V^*$ for almost every $x\in\Omega$
it holds 
\begin{equation}\label{eq:diffV}
\lim_{y\to x} \frac{\langle v^*, f(y) - f(x) - \nabla^{**}f(x)\cdot(y-x)\rangle}{|y-x|} = 0
\end{equation}
for all $v^*\in D^*$. 
Moreover there exists a sequence $(v^*_k)_{k\in\mathbb N}$ in $V^*$ such that
\[
\textnormal{md}(f, x)(\nu) = \sup_k\frac{\langle v^*, \nabla^{**}f(x)\cdot\nu\rangle}{\|v^*_k\|_{V^*}}  \quad \text{ for } \nu\in\mathbb R^n \textnormal{ and for almost every } x\in\Omega.
\] 
\end{thm}
\begin{proof}
Let $\varphi: V\to \mathbb R$ be a Lipschitz continuous function. 
Then for any $\nu\in\mathbb R^n$, $h>0$, and $x\in\Omega$
\[
\frac{|\varphi\circ f(x+h\nu) - \varphi\circ f(x)|}{|h|} \leq \operatorname{Lip}(\varphi)\operatorname{Lip}(f)|\nu|.
\]
Therefore due to Lemma \ref{lemma:lip-derivative} for every $x\in\Omega$ 
there exist a linear functional $\partial^\circ_\nu f(x)\in\left(\textnormal{Lip}_{0}(V)\right)^*$.
Because $V^*\subset \textnormal{Lip}_{0}(V)$ we can define a linear functional $\partial^{**}_{\nu}f(x)\in V^{**}$
as
\[
\big\langle v^*, \partial^{**}_\nu f(x)\big\rangle_{V^*,V^{**}} 
:= \big\langle v^*, \partial^\circ_\nu f(x) \big\rangle_{\textnormal{Lip}_{0}(V), \left(\textnormal{Lip}_{0}(V)\right)^*}.
\]
Set 
\[
\nabla^{**}f(x) = \big(\partial^{**}_{1}f(x), \dots, \partial^{**}_{n}f(x) \big).
\]

Let $\{v^*_k\}_{k\in \mathbb N}$ be a dense sequence in $D^*$.
Then $x\mapsto \langle v^*_k, f(x)\rangle$ is a Lipshitz continuous function on $\Omega$.
By the Rademacher theorem, there exists a null-set $N\subset\Omega$ 
such that for all $k\in\mathbb N$ function $\langle v^*_k, f(\cdot)\rangle$ is differentiable at any $x\in\Omega\setminus N$.
Moreover by Lemma \ref{lemma:lip-derivative} 
\[
\nabla \langle v^*_k, f(x)\rangle = \langle v^*_k, \nabla^{**}f(x)\rangle \quad \text{ for } x\in\Omega\setminus N.
\]
Therefore
\[
\lim_{y\to x} \frac{\langle v^*_k, f(y) - f(x) - \nabla^{**}f(x)\cdot(y-x)\rangle}{|y-x|} = 0
\]
for all  $v^*_k$ and $x\in\Omega\setminus N$.

Now let $v^*\in D^*$, then for any $\varepsilon>0$ there exist $v^*_k$ with $\|v^* -v^*_k \|_{V^*}\leq \varepsilon$.
So we have
\begin{multline*}
\frac{|\langle v^*, f(y) - f(x) - \nabla^{**}f(x)\cdot(y-x)\rangle|}{|y-x|}\\
\leq \varepsilon(1+\sqrt{n})\operatorname{Lip}(f) 
+ \frac{|\langle v^*_k, f(y) - f(x) - \nabla^{**}f(x)\cdot(y-x)\rangle|}{|y-x|}.
\end{multline*}
Thanks to the arbitrariness of $\varepsilon$ \eqref{eq:diffV} holds, and the theorem follows.
\end{proof}

\begin{thm} 
Let $V^*$ be a dual Banach space and let $f:\Omega\to V^*$ be a Lipschitz continuous mapping.
Then there exists $\nabla^{*}f :\Omega \to \mathcal B(\mathbb R^n; V^{*})$
such that for any separable subset $D\subset V$ for almost every $x\in\Omega$
it holds 
\begin{equation}
\lim_{y\to x} \frac{\langle v, f(y) - f(x) - \nabla^{*}f(x)\cdot(y-x)\rangle}{|y-x|} = 0
\end{equation}
for all $v\in D$. 
Moreover there exists a sequence $(v_k)_{k\in\mathbb N}$ in $V$ such that
\[
\textnormal{md}(f, x)(\nu) = \sup_k\frac{\langle v, \nabla^{*}f(x)\cdot\nu\rangle}{\|v_k\|_{V}}  \quad \text{ for } \nu\in\mathbb R^n \textnormal{ and for almost every } x\in\Omega.
\] 
\end{thm}

\begin{proof}
Again, the mapping $f$ satisfies the assumptions of Lemma \ref{lemma:lip-derivative}.
So there exist weak weak* partial derivatives  $\partial^\circ_\nu f(x)\in\left(\textnormal{Lip}_{0}(V^*)\right)^*$, $j=1,\dots, n$.
Define $\partial^{*}_j f:\Omega\to V^{*}$
by setting $\langle v, \partial^{*}_j f(x)\rangle := \langle J(v), \partial^\circ_\nu f(x) \rangle$, where $J:V\to V^{**}$ is the canonical embedding.
Then setting 
\[
\nabla^{*}f(x) = \big(\partial^{*}_{1}f(x), \dots, \partial^{*}_{n}f(x) \big)
\]
we could proceed as in the proof of Theorem \ref{thm:Vtarget}.

\end{proof}


\bibliographystyle{plain}
\bibliography{bibliography}
\end{document}